\documentclass[10pt]{amsart}
\usepackage{latexsym}
\usepackage[ansinew]{inputenc}

\usepackage{amsthm}
\usepackage{amsmath}
\usepackage{amssymb}
\usepackage{amsfonts}
\usepackage{color}
\usepackage{textcomp}

\usepackage{array}
\usepackage{float}
\usepackage{caption2}
\usepackage{graphicx}
\usepackage{anysize}
  
\usepackage{tabularx}
\newcommand{\beq}{\begin{equation}}
\newcommand{\enq}{\end{equation}}
\newtheorem{lemma}{Lemma}
\newtheorem{proposition}{Proposition}
\newtheorem{theorem}{Theorem}

\newtheorem{definition}{Definition}

\newcommand{\Z}{\mathbb{Z}}
\newcommand{\N}{\mathbb{N}}

\newcommand{\T}{T}

\newcommand{\E}{\tau_*}
\newcommand{\surf}{S_{0,\infty}}
\newcommand{\Sn}{\Sigma_{0,n}}

\newcommand{\B}{\mathcal{B}}

\newcommand{\Bd}{\mathcal{B}^{\frac{1}{2}}}
\newcommand{\Bext}{\widehat{\Bd}}

\newcommand{\C}{\mathcal{C_P}(\surf)}
\newcommand{\AC}{{\rm{Aut}}(\C)}

\newcommand{\LE}{{\mathcal{L}}(E)}
\newcommand{\LF}{{\mathcal{L}}(F)}
\newcommand{\LEA}{\mathcal{L}_{\surf}^A(E)}
\newcommand{\LFA}{\mathcal{L}_{\surf}^A(F)}
\newcommand{\LFn}{\mathcal{L}_{\surf}(F_n)}
\newcommand{\LEn}{\mathcal{L}_{\surf}(E_n)}

\title[Automorphisms of the asymptotic pants complex of an infinite surface]{On the automorphism group of the asymptotic pants complex
of an infinite surface of genus zero}
\author{
Louis Funar}
\address{Institut Fourier, UMR 5582, Laboratoire de Mathématiques,
Université Grenoble Alpes, CS 40700, 38058 Grenoble cedex 9, France, 
email:louis.funar@univ-grenoble-alpes.fr}
\author{Maxime Nguyen}
\address{Institut Fourier, UMR 5582, Laboratoire de Mathématiques,
Université Grenoble Alpes, CS 40700, 38058 Grenoble cedex 9, France, 
email: ngmaxime@gmail.com}

\date{}

\begin{document}

\maketitle

\begin{abstract}
The braided Thompson group $\B$ is an asymptotic mapping class group   
of a sphere punctured along the standard Cantor set, endowed with a rigid structure.
Inspired from the case of finite type surfaces we consider a Hatcher-Thurston 
cell complex whose vertices are asymptotically trivial pants decompositions. 
We prove that the automorphism group $\Bext$ of this complex is  
also an asymptotic mapping class group in a weaker sense.  Moreover $\Bext$ is obtained by 
$\B$ by  first adding  new elements called half-twists and further 
completing it. 
\end{abstract}

\textbf{2000 MSC classsification: 57 N 05, 20 F 38, 57 M 07, 20 F 34}



\section{Definitions and statements}
\subsection{Motivation}
Inspired by Royden's theorem on the holomorphic automorphisms 
of Teichm\"uller spaces Ivanov proved in \cite{ivanov} 
(subsequently completed by Korkmaz  and Luo \cite{korkmaz,luo}) that 
the automorphism group of the complex of curves of most compact surfaces 
coincides with the extended mapping class group.  
This was the start-point of many results of similar nature,
coming under the name of rigidity theorems. Margalit proved the 
rigidity  (see \cite{margalit}) of  pants complexes and further work 
extended this to even stronger rigidity 
theorem (see e.g. \cite{A,schmutz,luo,irmak,irmakKorkmaz,irmakMccarthy,
korkmazPapadopoulos} for a non-exhaustive list). 

The study of such automorphisms groups  in the pro-finite or pro-unipotent 
categories seems fundamental in Grothendieck's program.
For instance, although the pro-finite pants complexes are still rigid 
the automorphism group of the  corresponding pro-finite curve complexes 
is a version of the Grothendieck-Teichm\"uller group 
(\cite{L} and references there).

Simpler versions of this general question concern the 
solenoids, whose study was started in \cite{BPS}, 
and then infinite type surfaces corresponding to direct limits.   
The purpose of this article is to make progress in the second case using 
the formalism of asymptotically rigid homeomorphisms and braided Thompson 
groups developed in \cite{universalmcg,braidedthompson}.   
A previous result in this direction is the 
rigidity theorem proved in \cite{fossasnguyen} for an infinite type 
planar surface related to the Thompson group  $T$ (see \cite{CFP}). 

In this note we will consider an infinite surface obtained from the 
sphere by deleting the standard Cantor set from the equator. 
Since its mapping class group is a topological group, the authors of \cite{universalmcg} 
introduced a smaller subgroup $\B$ called  asymptotically rigid mapping class group, which was proved to be 
finitely presented.  As its name suggests, one restricts to mapping classes of those homeomorphisms which preserve 
an extra structure on the surface, but only outside of large enough compact sub-surfaces. 

There were different but closely related versions of such asymptotically rigid mapping class groups considered 
independently by Brin (\cite{brin2}) and Dehornoy (\cite{dehornoy1,dehornoy2}).   All of them are usually designed by 
the generic term of braided Thompson groups, as they occur as extensions of some Thompson group 
(see \cite{CFP}) by an inductive limit of mapping class groups, in particular by infinite braid groups.  

For instance, $\B$ is the extension of the 
Thompson group $V$ by an inductive limit of pure mapping class groups of holed spheres corresponding to 
an exhaustion of the sphere punctured along a Cantor set.  The extra structure
in its definition 
is needed to make unique the extension of a given homeomorphism defined on a compact sub-surface to the 
whole surface. The action at infinity of such a homeomorphism is an avatar of the action of the 
Thompson group $V$ on the Cantor set.

The novelty in the present setting is the appearance of some mild   
non-rigidity phenomenon of the corresponding Hatcher-Thurston asymptotic pants complex.
Nevertheless the group of automorphisms is still an asymptotic 
mapping class groups of the surface, but now the extra structure preserved is weakened.

\subsection{The surface $\surf$}\label{sectionsurfaceinf}

Let $\mathbb{D}^2$ be the (hyperbolic) disc and suppose that its boundary 
$\partial\mathbb{D}^2$ is 
parametrized by the unit interval (with  endpoints identified). 
Let $\E$ denote the (dyadic) Farey triangulation 
of $\mathbb{D}^2$. This triangulation is given by the family of bi-infinite geodesics representing the 
standard dyadic intervals, i.e. the family of geodesics $I_a^n$ joining the points $p=\frac{a}{2^n}$, 
$q=\frac{a+1}{2^n}$ on $\partial\mathbb{D}^2$, where $a,n$ are integers satisfying 
$0 \leq a \leq 2^n -1$. Let $T_3$ be the dual graph of $\E$, which is an infinite (unrooted) 
trivalent tree. Let $\Sigma$ be a closed $\delta$-neighborhood of $T_3$.

 Let $\surf$ be the infinite surface obtained from gluing two copies of $\Sigma$ along its boundary. 
We assume in addition that the family of arcs coming from the two copies of $\E$ defines a collection 
of simple closed curves, denoted by $E$.

 A {\em pants decomposition} of a surface  is a maximal collection of distinct homotopically nontrivial simple 
closed curves on it which are pairwise disjoint and non-isotopic. The complementary regions 
(which are 3-holed spheres) are called {\em pair of pants}.
The collection of simple closed curves $E$ is a {\em pants decomposition} of $\surf$. Subsequently, $E$ 
will be called the {\em standard pants decomposition} of $\surf$. 

\begin{definition}
 A {\em prerigid structure} on a surface is a collection of disjoint properly embedded  line segments, 
 such that the complement of their union  has two connected components.

 A {\em rigid structure} is the data consisting of a 
pants decomposition and a prerigid structure such that : 
\begin{enumerate}
 \item  The traces of the prerigid structure on each pair of pants are made of three 
connected components, called {\em seams};
 \item  For each pair of boundary circles of a given pair of pants, there is exactly one seam joining the two circles.

\end{enumerate}
\end{definition}

We will often call {\em rigid surface} a surface endowed with a rigid structure. 

We arbitrarily fix a prerigid structure associated to the standard pants decomposition $E$ to 
obtain a rigid structure which is called the \textit{standard rigid structure} of $\surf$.  The 
complement in $\surf$ of the union of lines of the canonical prerigid structure has two components: 
we distinguish one of them as the \textit{visible side} of $\surf$. Remark that the visible side 
of $\surf$ is homeomorphic to the initial surface $\Sigma$. 

\begin{figure}[ht!]
\begin{center}
\includegraphics[width=0.8\textwidth]{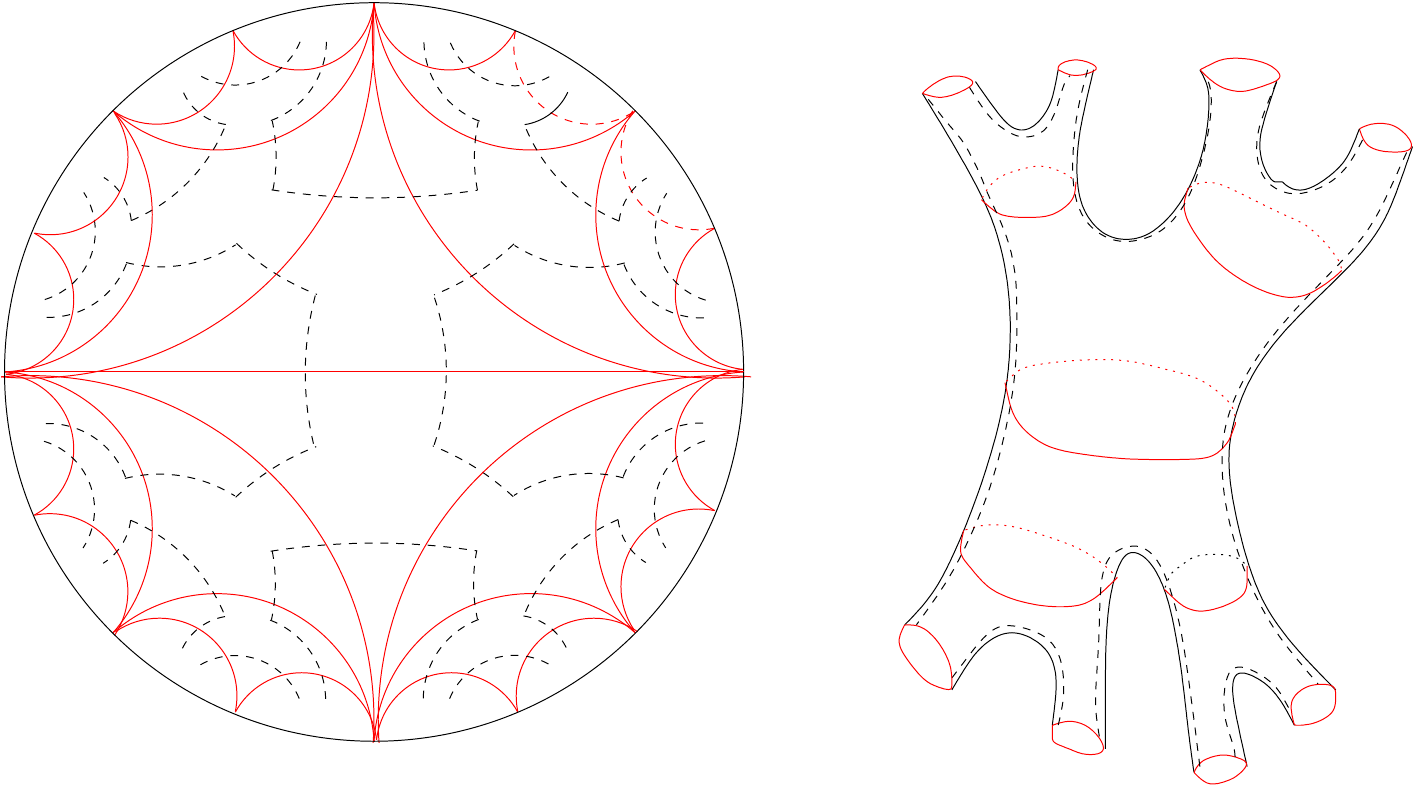}
\caption{\label{sinfdecompcouture}The surfaces $(\Sigma,\tau_*)$ and  
$(\surf,E)$}
\end{center}
\end{figure}

\subsection{The mapping class groups $\B$, $\Bd$ and $\Bext$}\label{sectionBplus}

\begin{definition}
 A compact sub-surface $\Sn \subset \surf$ (of genus zero with $n$ boundary components) is {\em almost admissible} if its boundary is contained in the standard pants decomposition $E$. 
 The {\em level} of a compact sub-surface is the number $n$ of 
its boundary components. 
\end{definition}

Observe that any almost admissible sub-surface $\Sn\subset \surf$ inherits 
a (standard) rigid structure by restricting the standard rigid structure. 

\begin{definition}
Consider an almost admissible sub-surface $\Sn \subset \surf$ endowed with an arbitrary rigid structure. 
We say that the rigid sub-surface $\Sn$ is {\em quasi-admissible} if the seams 
of $\Sn$ have the same endpoints as the seams of the standard rigid structure $E$ on $\surf-\Sn$, so 
that we can glue together the seams of $\Sn$ with the seams of the standard rigid structure on 
$\surf-\Sn$ to obtain a prerigid structure 
on $\surf$.   The visible side of $\surf$ induced from $\Sn$ is the one containing the visible side of $\Sn$.  
If the trace of this visible side on  $\surf-\Sn$ coincides with the trace of the 
visible side of the standard rigid structure $E$ we say that the rigid sub-surface $\Sn$ is {\em admissible}. 
\end{definition}

\begin{definition}
Let $f$ be a homeomorphism of $\surf$, endowed with some rigid structure. 
One says that $f$ is  {\em almost-rigid} if it stabilizes the  pants decomposition underlying the rigid structure.
Further, $f$ is {\em quasi-rigid} if it maps the pants decomposition into itself and the seams 
into the seams. If moreover $f$ sends the visible side into the visible side, then $f$ is said 
to be {\em rigid}. 
\end{definition}

Let $\Sn\subset \surf$ be a quasi-admissible rigid sub-surface and 
$f$ a homeomorphism of $\surf$. 
When $f(\Sn)$ is quasi-admissible, the image of the rigid structure of $\Sn$ by $f$ 
is also a rigid structure on $f(\Sn)$, which is said to be {\em induced by} $f$.

\begin{definition}
The homeomorphism $f$ of $\surf$ is {\em asymptotically  quasi-rigid} (resp.  
{\em asymptotically almost-rigid})  if  there exists an almost-admissible sub-surface 
$\Sn \subset \surf$, called a {\em support} of $f$, such that:
\begin{enumerate}
\item $f(\Sn)$, with the rigid structure induced by $f$ from the standard rigid structure, is  quasi-admissible; 
\item the restriction of $f:\surf - \Sn\to \surf-f(\Sn)$ is quasi-rigid (resp. almost-rigid). 
\end{enumerate}
If, moreover $f(\Sn)$ is admissible and $f:\surf - \Sn\to \surf-f(\Sn)$ is 
rigid, then we call $f$ {\em asymptotically rigid}. 
\end{definition}

\begin{definition}
The {\em asymptotic mapping class groups} $\Bd$, $\Bext$ and $\B$ of $\surf$ 
denote the groups of isotopy classes of asymptotically quasi-rigid, 
asymptotically almost-rigid and orientation preserving asymptotically rigid homeomorphisms, respectively. 
\end{definition}
 The group $\B$ appeared in \cite{universalmcg}, where it was called  
the universal mapping class group of genus zero. 

We will speak below of {\em asymptotically rigid ({\rm resp.} quasi-rigid)} mapping classes, as being 
isotopy classes of asymptotically rigid (resp. quasi-rigid) homeomorphisms.  

\subsection{The complex of pants decompositions of $\surf$}\label{sectioncomplexeinf}
An {\em asymptotically trivial pants decomposition} of $\surf$ is a pants decomposition which coincides with $E$ 
outside an admissible sub-surface. We define a cellular complex whose vertex set is the set of 
asymptotically trivial pants decompositions.

\begin{definition}
 Let $F$ and $F'$ be two asymptotically trivial pants decompositions of $\surf$. Let $c$ be a curve of 
$F$. We say that $F$ and $F'$ differ by an {\em elementary move} along the curve $c$ if $F'$ is obtained 
from $F$ by replacing $c$ by another curve which intersects $c$ twice and does not intersect the other
 curves of $F$. 
\end{definition}

\begin{figure}[ht!]
\begin{center}
\includegraphics[width=0.5\textwidth]{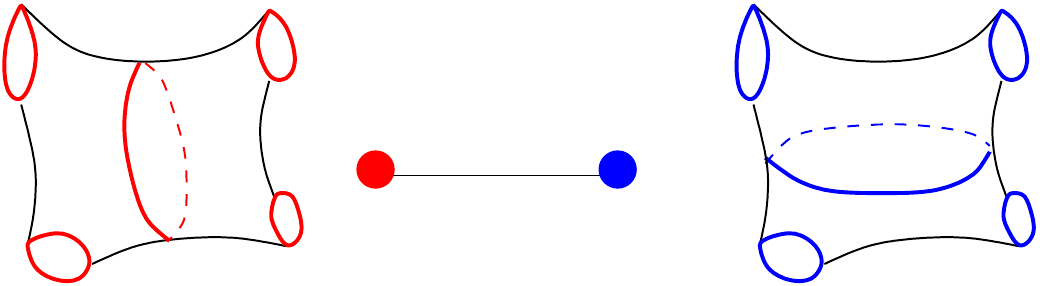}
\caption{\label{mvtelementaire} Elementary move}
\end{center}
\end{figure}

\begin{definition}
 Let $\C$ denote the {\em Hatcher-Thurston  pants complex} of $\surf$, defined in (\cite{universalmcg}, Def.5.1)
  as follows :
\begin{itemize}
 \item The vertices are the asymptotically trivial pants decompositions of $\surf$ ;
 \item The edges correspond to pairs of pants decomposition which differ by an elementary move ;
 \item The 2-cells are introduced to fill triangular cycles (see Fig. \ref{deuxcelltriangle}),
  square cycles corresponding to commutativity of moves with disjoint supports and pentagonal cycles 
  (see Fig. \ref{deuxcellpenta}).
\end{itemize}
\begin{figure}[ht!]
\begin{center}
\includegraphics[width=0.4\textwidth]{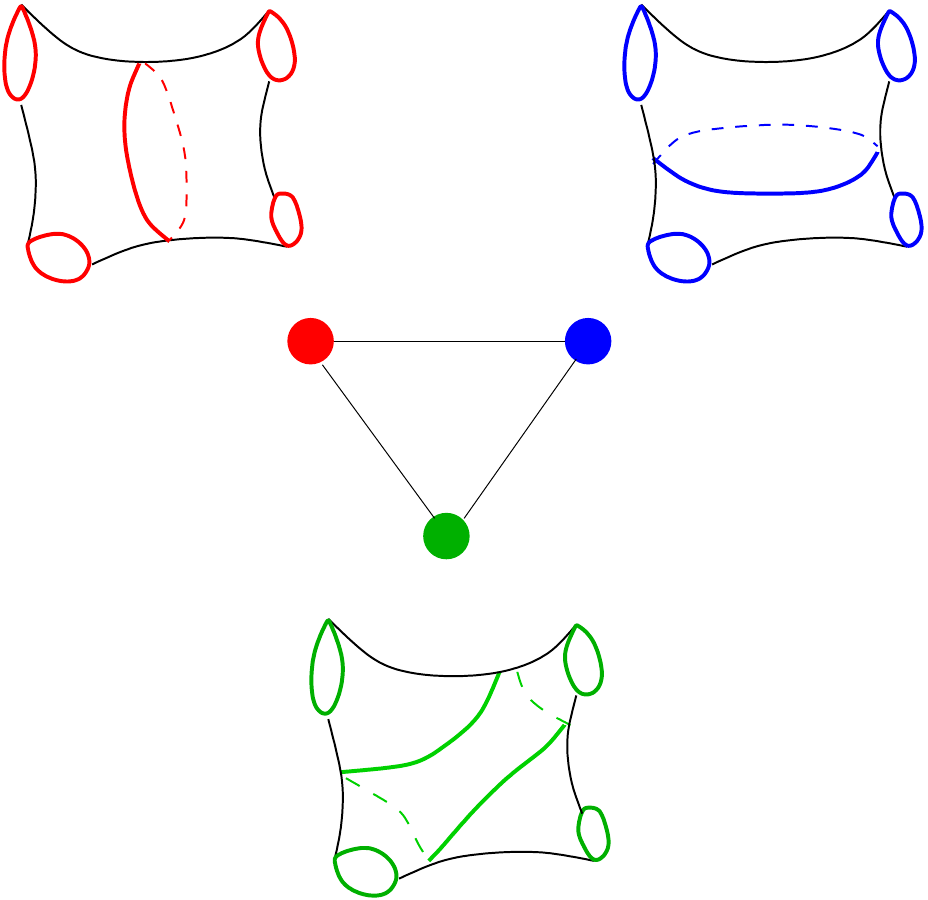}
\caption{\label{deuxcelltriangle} Triangular 2-cell in $\C$.}
\end{center}
\end{figure}

\begin{figure}[ht!]
\begin{center}
\includegraphics[width=0.5\textwidth]{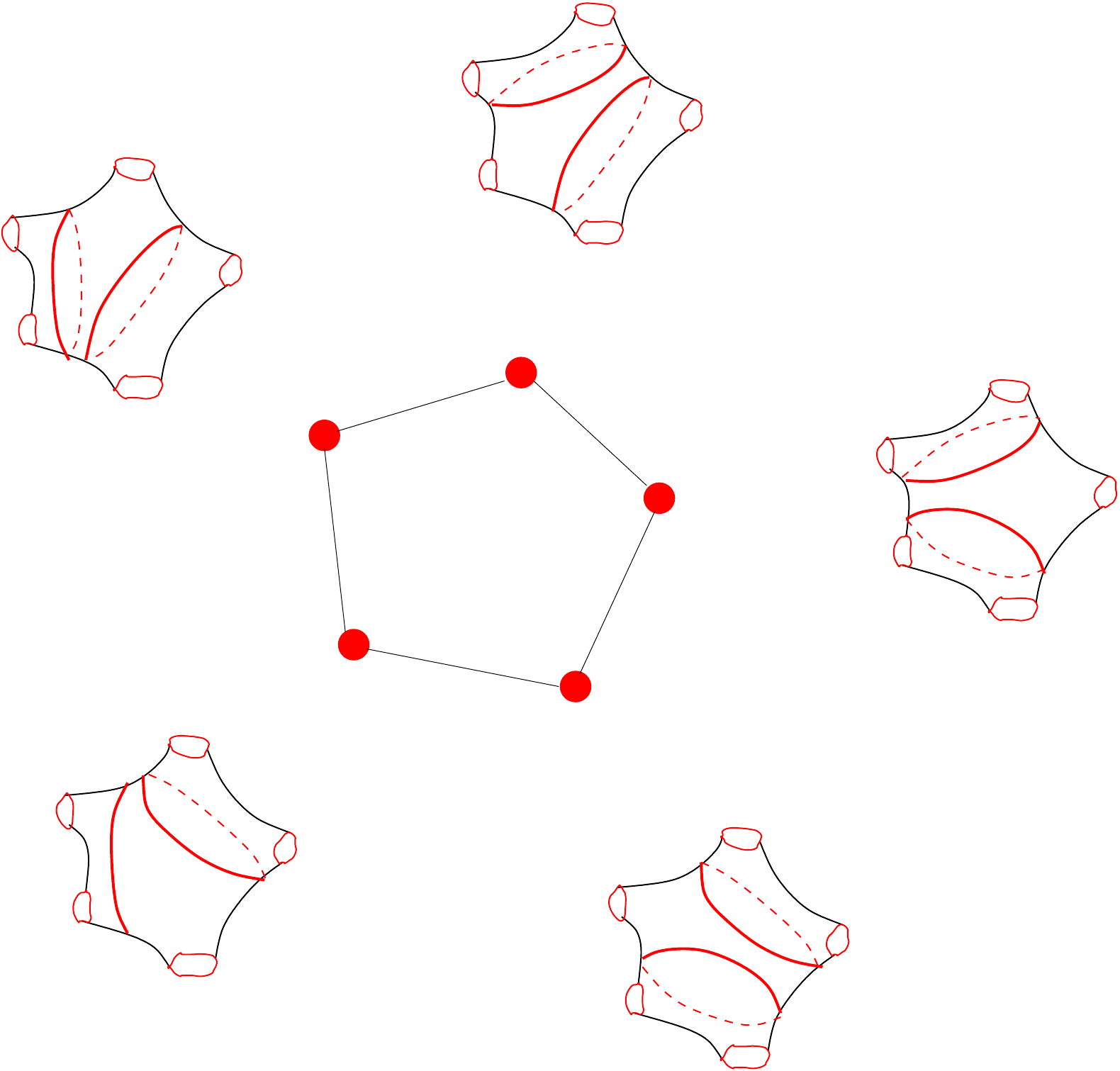}
\caption{\label{deuxcellpenta} Pentagonal 2-cell in $\C$.}
\end{center}
\end{figure}
\end{definition}

It is known (see \cite{universalmcg}, Prop.5.4) that $\C$ is connected and simply connected and 
$\B$ acts cellularly on it, with one   orbit of $0$-cells, one 
orbit of $1$-cells, one orbit of triangular $2$-cells, one orbit of pentagonal $2$-cells but infinitely many 
orbits of square $2$-cells. 

\subsection{The main result}

\begin{theorem}
 The asymptotic mapping class group $\Bext$ is isomorphic to the group of automorphisms of $\C$.
\end{theorem}

Before to proceed with the proof we will give more details about the groups $\Bd$ and $\Bext$ in section 
\ref{prelim}, by providing a finite generating set for $\Bd$ and explaining in which sense $\Bext$ is a 
completion of $\Bd$. We introduce the subgroups $D$ and $\widehat{D}$ of half-twists and give their structure, to be used 
later. In particular, one characterizes  the group of half-twists as the common 
stabilizer of the standard pants decomposition $E$ and a pair of pants on $\surf$. 

We prove in section \ref{decompositions} that $\Bext$ acts faithfully  on the complex $\C$, and thus it remains 
to show that every automorphism of $\C$ is induced by an element of $\Bext$. 
From an explicit geometric characterization of the 2-cells of $\C$ we deduce in section \ref{sectionautomorphismes} 
that two automorphisms of $\C$ which coincide on the vertices adjacent to the base vertex $E$ should be equal.  
One additional ingredient is the link graph $\LF$ of a vertex $F$ and its associated collapsed graph $\LFA$, whose vertices are in one-to-one correspondence with the curves in the pants decomposition $F$. We then study the action of half-twists, 
on these links.
 
The final steps of the proof of our main result are given in section \ref{cascompact}. Given an automorphism of $\C$ we 
construct an element of $\Bd$ with the same action on the restricted link of $E$ associated to a compact sub-surface. 
By induction we obtain a sequence of elements in $\Bd$ which coincide with the given automorphism on sub-sets 
of $\LE$ corresponding to bigger and bigger compact sub-surfaces. The associated infinite product 
is an element of $\Bext$ whose action on $\LE$, and hence on all of $\C$ is the prescribed one.

\section{Preliminaries}\label{prelim}

\subsubsection*{Thompson groups and $\B$}

There is a natural projection $\pi : \surf \rightarrow \Sigma$, where $\Sigma$ was defined in section \ref{sectionsurfaceinf}, such that the pullback of the arcs 
of $\E$ is the set of closed curves of $E$. Then we have a bijection  between the set of 
standard dyadic intervals and the set of closed curves of $E$:

$$\{\text{curves of } E\} \leftrightarrow \{\text{arcs of } \E\} \leftrightarrow \{\text{standard dyadic intervals}\} $$

Recall that the Thompson group $V$ is the group of right-continuous bijections of $S^1$ that map images of 
dyadic rational numbers to images of dyadic rational numbers, that are differentiable except at 
finitely many images of dyadic rational numbers, and such that, on each maximal interval on which 
the function is differentiable, the function is linear with derivative a power of 2.

The Thompson group $T$ is the group of piecewise linear homeomorphisms of $S^1$ that map images of 
dyadic rational numbers to images of dyadic rational numbers, that are differentiable except at 
finitely many images of dyadic rational numbers and on intervals of differentiability on which the 
function is differentiable, the function is linear with derivative a power of 2.

For more details about Thompson groups, see \cite{CFP}. Furthermore, we know  (see \cite{braidedthompson}) that $T$ can be 
viewed as an asymptotic mapping class group of the planar surface $\Sigma$. 

Let $K_{\infty}^*$ be the inductive limit $\bigcup_{n=0}^{\infty} K^*(3\cdot 2^n)$ where $K^*(n)$ is the 
pure mapping class group of the $n$-holed sphere. We have the exact sequence (\cite{universalmcg}) :

$$ 1 \longrightarrow K_{\infty}^* \longrightarrow \B \longrightarrow V \longrightarrow 1$$

Moreover, we have the following result from (\cite{universalmcg}, proof of Prop. 2.4):

\begin{proposition}\label{propvisibleT}
 The Thompson group $T$ is the subgroup of elements of $\B$ which preserve the visible side of $\surf$. 
\end{proposition}

\subsubsection*{Generators of $\B$}
We fix an admissible pair of pants $P$ of $\surf$ to be called the {\em fundamental pair of pants}.
The surface $\surf$ retracts onto a tree which is the adjacency graph of the standard pants decomposition $E$. 
Choosing a fundamental pair of pants amounts to enhance this tree to a {\em rooted tree} $\mathcal T$. 
In particular, curves of $E$ are in one-to-one correspondence with the mid-points of the edges of $\mathcal T$. 
Similarly,  the {\em fundamental four-holed sphere} is the union of two adjacent admissible pairs of pants, one of them 
being the fundamental one.
 
The group $\B$ is generated by the twist $t$  (see Fig \ref{twist2}), the braid $\pi$ 
(Fig  \ref{tresse}) and the lifts $\alpha$  (see Fig \ref{alpha}) and $\beta$ (see Fig \ref{beta}) 
of the two generators of $T$ which are usually denoted by the same letters. For more details see section 3 of \cite{universalmcg}.

\begin{figure}[ht!]
\begin{center}
\includegraphics[width=0.5\textwidth]{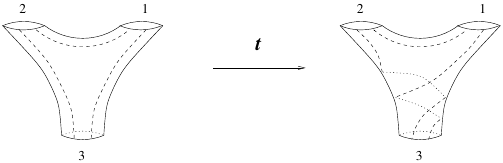}
\caption{\label{twist2}The action of a twist $t$ on the fundamental pair of pants.}
\end{center}
\end{figure}

\begin{figure}[ht!]
\begin{center}
\includegraphics[width=0.5\textwidth]{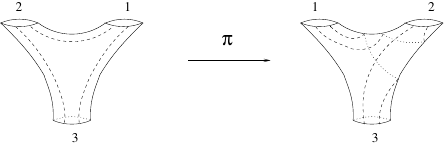}
\caption{\label{tresse}The action of a braid $\pi$ on the fundamental pair of pants.}
\end{center}
\end{figure}

\begin{figure}[ht!]
\begin{center}
\includegraphics[width=0.5\textwidth]{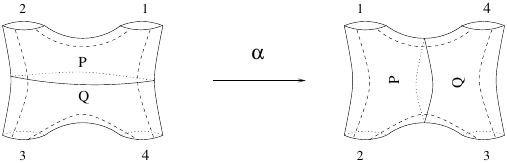}
\caption{\label{alpha}The action of $\alpha$ on the fundamental four-holed sphere.}
\end{center}
\end{figure}

\begin{figure}[ht!]
\begin{center}
\includegraphics[width=0.5\textwidth]{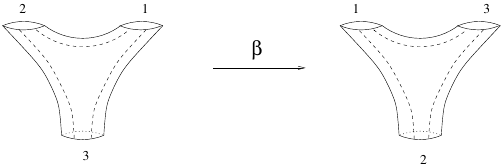}
\caption{\label{beta}The action of a braid $\beta$ on the fundamental pair of pants.}
\end{center}
\end{figure}

Recall that the subgroup of $\B$ consisting of those mapping classes represented by 
rigid homeomorphisms is isomorphic to $PSL_2(\Z)$ (see \cite{universalmcg}, Remark 2.1). 
The action of $PSL(2,\Z)$ on the rooted tree $\mathcal T$ is transitive. In particular, for any 
curve $a$ of $E$ the Dehn twist $t_a$ along the curve $a$ is a conjugate of $t$ in $\B$. 

Let $\Sn$ be a $n$-holed sphere embedded in $\surf$ such that its 
boundary components lie in 
the standard decomposition $E$ of $\surf$. Recall that the mapping 
class group $\mathcal{M}(\Sn)$ 
is the group of isotopy classes of homeomorphisms of $\Sn$ preserving 
the orientation. The elements 
of $\mathcal M(\Sn)$ are represented by homeomorphisms which can permute the boundary 
components of $\Sn$. We can 
assume that they also preserve the trace of these boundary components on 
the visible side. Then there 
exists an embedding $\mathcal{M}(\Sn) \rightarrow \B$ obtained by extending 
rigidly a homeomorphism  
representing a mapping class of $\mathcal{M}(\Sn)$.

\subsection{The extended mapping class group $\Bext$}
We now define some mapping classes which are not represented by 
asymptotically rigid homeomorphisms. 
These elements will generate the extended mapping class group $\Bd$, whose completion is  $\Bext$.

\subsubsection{The symmetry}

\begin{definition}
Let $i_R$ be the isotopy class of the quasi-rigid homeomorphism acting as the symmetry on $\surf$ 
which permute the visible side and the invisible side on each pair of pants of the standard 
decomposition $E$ and preserves the seams.  
\end{definition}

\subsubsection{The half-twists}
Let $a$ be a curve of the standard decomposition $E$. 
Using the metric on the rooted tree $\mathcal T$ we can speak about the distance between 
two curves of $E$ and hence between subsets of $E$. There exists then a unique  admissible 
pair of pants $P_a$ which is different from the fundamental pair of pants $P$,  
such that $a$ is the closest among the three boundary components of $P_a$ 
to the fundamental pair of pants. 

Let $a,b,c$ be the bounding circles of  $P_a$. By cutting along each of these three curves, we define three 
infinite connected components of $\surf$ respectively denoted by 
$S_a,S_b,S_c$ and one pair of pants $P_{a}$, such that $S_a$ contains the fundamental pair of pants. 
Let $D_{a}$ be a homeomorphism of $P_{a}$ such that 
$D_{a}$ fixes $a$ and interchanges the other two curves, i.e.   
$D_{a}(b) = c$ and $D_{a}(c)=b$, by means of a rotation of 
angle $\pi$. By definition the trace of $b$ on the visible side is sent 
by $D_{a}$ to the trace of $c$ on the invisible side 
and the trace of $c$ on the visible side is sent onto the trace of $b$ on the invisible 
side. Then we quasi-rigidly extend this homeomorphism on $\surf$. 
This means that $D_{a}$ acts on $S_a$ as  identity, sending the visible side of 
$S_b$ on the invisible side of $S_c$, the invisible side of 
$S_b$ on the visible side of $S_c$, the visible side of $S_c$ on the invisible side of $S_b$ and  the 
invisible side of $S_c$ on the visible side of $S_b$. 

\begin{definition}\label{halftwist}
We denote the isotopy class  of 
the homeomorphism $D_{a}$ by $d_a$ and call it the  
{\em half-twist along  $a$}.
\end{definition}
 We denote by 
$supp(d_{a})$ the sub-surface of $\surf$ 
on which $D_{a}$ acts non trivially up to homotopy, which is larger than its smallest support $P_a$. 
With the previous notations, we 
have $supp(d_{a}) = P_{a} \cup S_b \cup S_c$.

To emphasize the relation between $\Bd$ and the Thompson-type groups  
note that each curve of the standard decomposition $E$ is associated to a standard dyadic interval 
of the form $I_k^n:=[\frac{k}{2^n},\frac{k+1}{2^n}]$ where $k$ and $n$ are integers such 
that $0 \leq k \leq 2^n-1$. Moreover, the standard dyadic intervals $I_{2k}^{n+1}$ and $I_{2k+1}^{n+1}$ 
both define with $I_k^n$ a pair of pants in $E$. Hence, each half-twist along a curve belonging in $E$ 
is defined by a couple $(k,n) \in \N^2$ such that $0 \leq k \leq 2^n-1$ and we can also use the 
notation $d_a=d_{I_k^n}$. 
For all $n\leq m$ and $j,k$ satisfying  $0 \leq j \leq 2^n-1$ and  $0 \leq k \leq 2^m-1$, we 
have $supp(d_{I_j^n}) \cap supp(d_{I_k^m}) = \emptyset$ or $supp(d_{I_j^n}) \subset supp(d_{I_k^m})$. 
Therefore, for all sequences $(j_i,n_i,p_i)_{i \in \N}$ of elements of $\N \times \N \times \Z$ 
such that 
for all $i \in \N$, $ 0 \leq j_i \leq 2^{n_i}-1$ the sequence $(n_i)_{i \in \N}$ is non decreasing,
we can define the infinite product $\prod_{i=0}^{\infty} (d_{I_{j_i}^{n_i}})^{p_i}$ as an element of 
the mapping class group of the surface $\surf$. 

\begin{figure}[ht!]
\begin{center}
\includegraphics[width=0.4\textwidth]{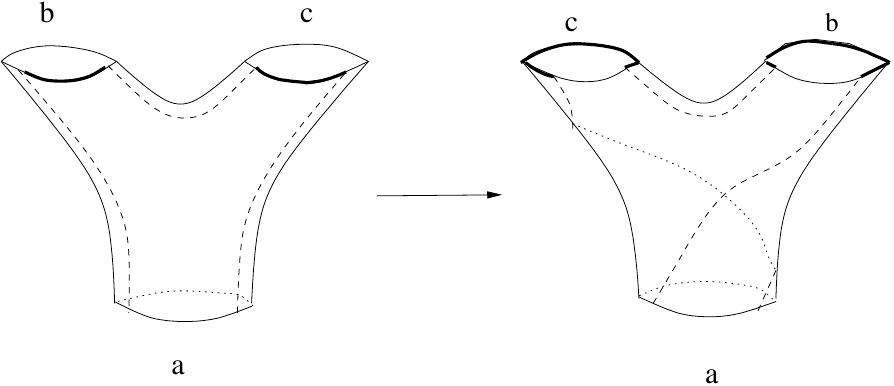}
\caption{\label{demitwist1}Action of a half-twist $d_{a}$ on the pair of pants $P_a$.}
\end{center}
\end{figure}


\subsubsection{The group $\Bd$}

\begin{proposition}\label{generators}
 The group $\Bd$ is  generated by one half-twist around 
some curve of the standard 
decomposition $E$, the symmetry $i_R$ and the elements $\alpha$ and $\beta$ of the 
group $\B$.
\end{proposition}

{\em Convention}. By a slight abuse of language we speak about  the image 
$g (S)$, where $g$ is an isotopy class of a 
homeomorphism which is rigid on $\surf-S$, 
by letting it be the admissible sub-surface image of $S$
by means of some homeomorphism representing $g$. 

\begin{proof}
Let $g\in \Bd$. By possibly composing with the symmetry $i_R$ we can assume that $g$ 
is the mapping class of an orientation preserving homeomorphism of $\surf$. 
There exists an admissible standard sub-surface $S \subset \surf$,  with $g(S)$ quasi-admissible such that $g$ 
is quasi-rigid outside $S$. There is no loss of generality in assuming that $S$ contains the fundamental 
pair of pants $P$. Let $a_j$, $1\leq j\leq n$, be those boundary components of $S$ such that the visible side induced from 
$g(S)$ on  the connected component of $\surf-S$ containing $a_j$ does not agree with the 
trace of the standard rigid structure $E$. 
Replace now $S$ by  $S'=S\cup\bigcup_{j=1}^nP_{a_j}$. Let $a_{j,0}, a_{j,1}$ and $a_j$  be the 
boundary components of  $P_{a_j}$. 
Since $g$ is quasi-rigid the element $gd_{a_1}d_{a_2}\cdots d_{a_n}$ is also quasi-rigid. 
But now the visible sides of  the connected components of 
$\surf-S'$  containing either $a_{j,0}$ or $a_{j,1}$, which are induced from $g d_{a_1}d_{a_2}\cdots d_{a_n}(S')$ 
agree with the trace of the standard rigid structure, since the half-twist exchange the visible and 
invisible sides of the permuted boundary components. The remaining boundary components 
of $S$ correspond to those of $S'$. It follows that  $g d_{a_1}d_{a_2}\cdots d_{a_n}(S')$ is admissible and 
hence  the mapping class $g d_{a_1}d_{a_2}\cdots d_{a_n}$ is  rigid outside $S'$ and hence 
asymptotically rigid.  This means that $g d_{a_1}d_{a_2}\cdots d_{a_n}\in \B$. 
Therefore $\Bd$ is generated by $\B$ and the set of half-twists along curves in $E$.

Further observe  that the twist  $t_a$ is given by $t_a=d_a^2$.   
Note that a half-twist along any curve of $\surf$ is conjugate 
to a half-twist along a curve of $E$ by an element of $\B$, since $\B$ acts 
transitively on the set of isotopy classes of simple closed curves of $\surf$. 

Let now $c$ be a curve of $E$.  Set $\pi_c$  for the conjugate of $\pi$ which fixes $c$ and acts in the same way on the 
pair of pants $P_c$ as $\pi$ does on the fundamental pair of pants. 
Denote by  $c_0$, $c_1$ the two other boundary components of $P_c$, 
so that $c,c_0,c_1$ is a  clockwise oriented triple, with respect to the local cyclic order structure 
induced by the planar visible part. 
We  claim that: 

\begin{lemma}\label{pi}
We have $\pi_c=d_cd_{c_0}^{-1}d_{c_1}^{-1}$. 
\end{lemma}

We can obtain $d_c, d_{c_0}$ and $d_{c_1}$ as conjugates of a given half-twist by elements of $T\subset \B$. 
Assuming this Lemma,   it follows that $\Bd$ is generated by $T$ (which is generated by $\alpha$ and $\beta$), 
$i_R$ and 
one half-twist, as claimed. 
\end{proof}

\begin{proof}[Proof of Lemma \ref{pi}]
Let us denote by  $c_{i0}$, $c_{i1}$ the two other boundary components of $P_{c_i}$, with the same convention as above. 
To compare these two elements we have to consider  
a larger sub-surface of level 4, see the 
following picture: 
 
\begin{center}
\includegraphics[width=0.6\textwidth]{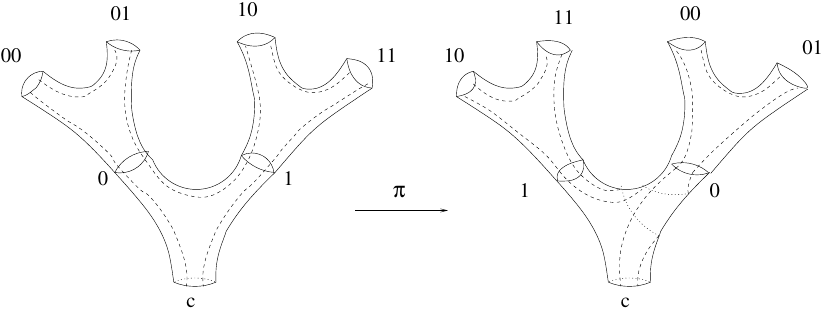}
\end{center}

The composition of half-twists $d_cd_{c_0}^{-1}d_{c_1}^{-1}$
on a sub-surface of level 4 is viewed below: 

\begin{center}
\includegraphics[width=0.6\textwidth]{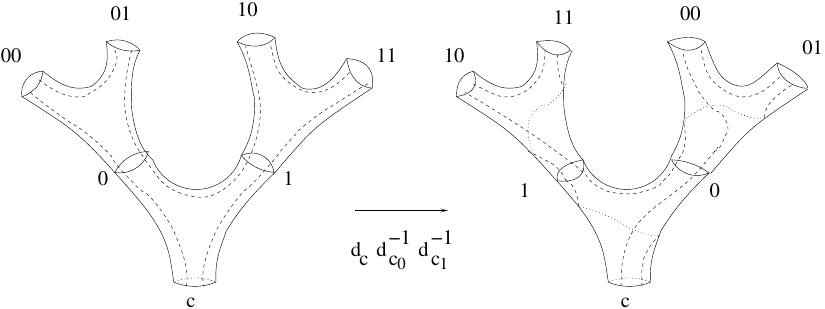}
\end{center}

Observe that the visible sides of the curves $c_0$ and $c_1$ are sent by 
$\pi_c$ into the visible sides of the curves $c_1$ and $c_0$, respectively. 
On the other hand,  the visible sides of the curves $c_0$ and $c_1$ are sent by 
$d_cd_{c_0}^{-1}d_{c_1}^{-1}$ into the invisible sides of the curves $c_1$ and $c_0$, respectively. 

Nevertheless, there exists an isotopy  of $\surf$ which is identity outside the 
sub-surface $\Sigma_{0,4}$ sending the rigid structure induced by $\pi_c$ into 
the one induced by $d_cd_{c_0}^{-1}d_{c_1}^{-1}$:     
twist along each $c_0$ and $c_1$ counterclockwisely (with respect to the boundary orientation) 
until  the endpoints of the seams 
are switched, thereby exchanging the visible and the invisible part 
of these circles. This proves that these two elements of $\Bd$ coincide. 
\end{proof}

\subsubsection{The group $\Bext$}\label{sectionBext}
We will show now that $\Bext$ is obtained from $\Bd$ by a process of passing to limit. 

\begin{proposition}\label{normal}
Every $g \in \Bext$ can be written as an infinite product:  
\[ g = f \cdot  \prod_{i=1}^{\infty} d_{c_i}^{p_i}\]
where:
\begin{enumerate}
\item $f \in \Bd$ acts quasi-rigidly outside some admissible  sub-surface $S\subset \surf$;
\item $c_i$ is a sequence of curves belonging to $E\cap (\surf-S)$ which goes to infinity, i.e.  
which eventually leaves any compact sub-surface of $\surf$ and $p_i\in\Z$. 
\end{enumerate}
\end{proposition}
Note that both the left hand side  and the infinite product in the right hand side 
are well-defined elements in the mapping  class group of (all homeomorphisms of) $\surf$. 
\begin{proof}
There exists a quasi-admissible sub-surface $S \subset \surf$,  with $g(S)$ quasi-admissible, such that $g$ 
preserves globally the pants decomposition $E\cap (\surf-S)$ of $\surf-S$, namely: 
$g(E\cap  (\surf-S))=E\cap (\surf-g(S))$. There is no loss of generality in assuming that 
$S$ contains the fundamental pair of pants. 
Denote by $f$ the element of $\Bd$ which acts as $g$ on $S$ and is extended quasi-rigidly on $\surf-S$. 
Note that $f$ might not belong to $\B$ since the boundary circles do not necessarily inherit the 
right decomposition into visible and invisible part from $g$. 
It follows that $f^{-1}g$ stabilizes every connected component of $\surf-S$. 

Let us stick for the moment to the restriction to one connected component. 
Let $U$ be a connected component of $\surf-S$ and $y\in \Bext$ be an  element of its stabilizer.
The unique boundary circle $c$ of $U$ has to be sent into itself by $y$, and its visible part is sent into its visible part. Let $c_0, c_1$ be the two other curves which occur along with $c$ 
as the boundary of the pair of pants $P_c$ in $E\cap U$. Since $y$ is the mapping class of a homeomorphism, 
the curves $y(c_0), y(c_1)$ and $y(c)=c$ belong to $E\cap U$ and bound a pair of pants of $E$. 
This means that  $y(P_c)=P_c$ and thus 
$\{y(c_0),y(c_1)\}=\{c_0,c_1\}$. Up to composition with $d_c$ we can therefore assume that 
$y(c_i)=c_i$, for $i=0,1$.

If the visible parts of $c_i$ and $y(c_i)$ agree then, up to composition 
with  some powers of $t_c$, $t_{c_0}$ and $t_{c_1}$ we can assume that the restriction 
of $y$ to $P_c$ is identity, namely it sends the seams into seams and the visible part into the visible part. 

If, moreover the visible part of $c_i$ and $y(c_i)$ disagree 
then we compose $y$ with $d_{c_i}$ and we reduce ourselves to the previous situation.  
In order to see this, recall that  $P_{c_i}$ denotes the  pair of pants of $E\cap U$ which is  
adjacent to $P_c$ along $c_i$.  Then the composition of $y$ and  those $d_{c_i}$ needed above  
is the mapping class  of a homeomorphism of  $P_c\cup P_{c_1}\cup P_{c_2}$ which is 
isotopic  rel boundary to  one whose restriction to $P_c$ sends the visible part of $c_i$ into the visible part of $c_i$ 
(see the figures used in the proof of Lemma \ref{pi}).

Note that $d_{c_0}$ and $d_{c_1}$ commute with each other.  
Therefore, in both cases, by composing $y$ with some element of the form 
$t_c^{-n}d_{c_0}^{-n_0}d_{c_1}^{-n_1}$  we can make the restriction of $y$ to $P_c$ to be  identity. 

Now, for any finite sequence $I$ with entries from $\{0,1\}$, we denote by 
$c_{I0}$ and $c_{I1}$ the two curves at unit distance from the 
curve $c_I$ (using the metric induced from the tree $\mathcal T$) which are farther from $c$ than $c_I$.  Let $\|I\|$ denote the number of entries of $I$. 
By recurrence on $\|I\|$,  for any $k$ there exists $n_i\in \Z$, for $i\leq k$, such that 
the restriction of $yt_c^{-n}\prod_{\|I\|=1}^kd_{c_{I}}^{-n_I}$ to the admissible sub-surface 
bounded by $c$ and all the curves $c_I$, where $\|I\|=k$,  is identity.  
It follows that we can write in $\Bext$: 
\[ y=t_c^n\prod_{\|I\|=1}^{\infty} d_{c_I}^{n_I}\]

Notice that $d_{c_I}$ and $d_{c_J}$ commute when $\|I\|=\|J\|$ and more generally, 
if $supp(d_{c_I})\cap supp(d_{c_J})=\emptyset$. 

We can now resume  the proof of the statement. Set  $U_1,U_2,\ldots,U_n$ for
the connected components of $\surf-S$, whose boundary circles are denotes by $a_j$. From above we can write:
\[ f^{-1}g|_{U_j}=t_{a_j}^{n_j}\prod_{\|I\|=1}^{\infty} d_{a_{j, I}}^{n_{j,I}}, \; {\rm for} \; 1\leq j\leq n\]
Since $c_j$ are disjoint and $S$ contains the fundamental pair of pants we have 
 $supp(d_{c_j})\cap supp(d_{c_k})=\emptyset$, for $j\neq k$. 
 This implies that:
 $supp(d_{c_{j,I}})\cap supp(d_{c_{k,J}})=\emptyset$, for $j\neq k$, because 
 $supp(d_{c_{j,I}})\subset supp(d_{c_j})$.  We obtain therefore the identity: 
 \[ g= f \prod_{j=1}^n t_{a_j}^{n_j} \prod_{\|I\|=1}^{\infty} \prod_{j=1}^nd_{a_{j, I}}^{n_{j,I}}\]
 which proves the claim. 
  \end{proof}

\subsubsection{The sub-groups $D$ and $\widehat{D}$}\label{DD}
The {\em sub-group  of standard  half-twists} $D\subset \Bd$ is the subgroup generated by the  half-twists $d_a$, where 
$a$ runs over the set of curves in $E$. Note that the extended supports $supp(d_a)$ are all contained in the 
complement of the fundamental pair of pants, so that  whenever 
$a,b\in E$ then $supp(d_a)$ and $supp(d_b)$ are either disjoint 
or else one of them is contained in the other.  In particular, 
although one might think that it is reasonable to call the element $\alpha^2d_a\alpha^2$ (where $\alpha$ is the 
element of $\B$ from picture \ref{alpha})
also a  standard half-twist, it does not belong to the sub-group of standard half-twists $D$. 

Let $\widehat D\subset\Bext$ denote the sub-group consisting of possibly infinite products 
\[ d= d_{a_1}d_{a_2}\cdots d_{a_n}\cdots \]
of the half-twists $d_{a_i}$, where $a_i$ is a sequence of curves  from $E$ which goes to infinity, i.e. 
which eventually leaves every compact sub-surface of $\surf$. 

We wish to emphasize the fact that the sub-groups $D$ and $\widehat{D}$ depend on the choice 
of a fundamental pair of pants, up to conjugacy, although their isomorphism type does not. 

Let  $Perm^3_{2^n}$  denote the group of those automorphisms of a rooted trivalent tree
of level  $n+1$ which fix the neighbours of the root, or equivalently, the automorphism group 
of three copies of the  rooted binary tree of level $n$. The action of $Perm^3_{2^n}$ on the set of $3\cdot 2^n$ 
boundary leaves is faithful.  Denote by 
$Perm^3_{\infty}$ the inductive limit $\lim_{n\to \infty} Perm^3_{2^n}$, where 
$Perm^3_{2^n}\to Perm^3_{2^{n+1}}$ is induced by the embedding of the corresponding rooted trees.  

\begin{proposition}\label{ddemitwists}
The abelian sub-group $D[2]\subset D$ generated by the twists $t_a=d_a^2$ along the 
curves $a$ in $E$ is a normal subgroup of $D$ which fits into the exact sequence:
\[ 1\to D[2] \to D \to Perm^3_{\infty}\to 1\]
\end{proposition}
\begin{proof}
The homomorphism $D \to Perm^3_{n}$ corresponds to the action of $D$ on the trivalent sub-tree $\mathcal T_n$ 
of $\mathcal T$ whose vertices are those curves in $E$ at distance  at most $n$ from the fundamental pair of pants $P$.
 Each element of $Perm^3_{n}$ is a product of transpositions. Here by transposition we mean 
the transformation exchanging two branches having a common vertex by means of a planar symmetry, so that the cyclic order of their leaves is reversed. Every transposition is the image of the half-twist along the curve in $E$ 
corresponding to the vertex, so that   $D \to Perm^3_{n}$ is surjective for every $n$ and hence also 
for $n=\infty$.  

We show further that $D[2]$ is a normal subgroup of $D$.
By direct inspection we find that:  
\[ d_c d_a = \left\{\begin{array}{ll}
d_a d_c, &  \; {\rm if } \; supp(d_a)\cap supp(d_c)=\emptyset\\
d_{d_c(a)} d_c, &   \; {\rm if } \; supp(d_a)\subset supp(d_c)\\
d_a d_{d_a^{-1}(c)}, & \; {\rm if } \; supp(d_c)\subset supp(d_a)\\
\end{array}
\right.
\]
We derive that:
\[ d_c d_a^2 d_c^{-1}= \left\{\begin{array}{ll}
d_a^2, &  \; {\rm if } \; supp(d_a)\cap supp(d_c)=\emptyset\\
d_{d_c(a)}^2, &   \; {\rm if } \; supp(d_a)\subset supp(d_c)\\
d_a^2, & \; {\rm if } \; supp(d_c)\subset supp(d_a)\\
\end{array}
\right.
\]
and our claim follows.

Let now $g\in \ker(D \to Perm^3_{n})$ be a product of half-twists along curves of $E\cap \Sigma_{0,3\cdot 2^n}$, where 
$\Sigma_{0,3\cdot 2^n}$ is the surface whose boundary circles are at distance $n$ from the fundamental 
pair of pants $P$. Since $g$ is a homeomorphism preserving $E$ and fixing the boundary of $P$, 
$g$ should send $\Sigma_{0,3\cdot 2^k}$ into itself, for every $k$. 
By induction on $k$, the action of $g$ on the set of boundary components of $\Sigma_{0,3\cdot 2^k}$ 
must be trivial, for any $k\geq 0$. The induction step follows from the fact that 
 $\Sigma_{3\cdot 2^{k+1}}-\Sigma_{3\cdot 2^{k}}$ is a union of disjoint 
 pair of pants. Then the restriction of $g$ to each such pair of pants either sends 
 the new level $k+1$ boundary components into themselves, or else it permutes them. 
But a non-trivial permutation of them induces a non-trivial element in $Perm^3_{\infty}$. 
This proves that $g$ keeps fixed any curve of $E$ and hence all pair of pants in $\Sigma_{0,3\cdot 2^n}$. Since the mapping class group of a pair of pants is the abelian group generated by the three boundary Dehn twists, it follows 
that $g$ belongs to  the normal subgroup of $D$ generated by the Dehn twists along curves of $E$, i.e. 
to $D[2]$. 
\end{proof}

\begin{proposition}\label{extdemitwists}
Set  $\widehat{D[2]}\subset\Bext$ for  the sub-group consisting of possibly infinite products 
\[ d= t_{a_1}t_{a_2}\cdots t_{a_n}\cdots \]
of  Dehn twists $t_{a_i}$, where $a_i$ is a sequence of curves  from $E$ which goes to infinity.
Then  $\widehat{D[2]}$ is a normal subgroup of $\widehat D$ which fits into the exact sequence:
\[ 1\to \widehat{D[2]} \to \widehat D \to Perm^3_{\infty}\to 1\]
where $Perm^3_{\infty}$ is as above. 
\end{proposition}
\begin{proof}
Note first that  any element $d\in \widehat{D}$  can be written as an infinite product of the form: 
\[ d= d_{a_1}^{p_1}d_{a_2}^{p_2}\cdots d_{a_n}^{p_n}\cdots \]
of powers of half-twists $d_{a_i}$ along curves from $E$, where $p_j\in \Z$, 
so that there exist $k_0\leq k_1\leq \cdots \leq k_n\leq \cdots$ with the properties:  
\begin{enumerate}
\item $a_{k_n}, a_{k_{n}+1}, \ldots a_{k_{n+1}-1}$ are curves of level $n$, for each $n$; 
\item there are no more than $3\cdot 2^n$ curves of level $n$, i.e. 
$k_{n+1}-k_n\leq 3\cdot 2^n$, for all $n$. 
\end{enumerate}
Then the proof given above for Proposition \ref{ddemitwists} works without essential modifications. 
We skip the details. 
\end{proof}

\begin{proposition}\label{Dstabilizer}
Let $g\in \Bext$ be an element such that $g(E)=E$ and $g$ fixes the fundamental pair of pants $P$, namely 
$g$ is the mapping class of a homeomorphism whose restriction to $P$ is identity. 
Then $g\in \widehat{D}$. In particular, if  additionally $g\in \Bd$,  then $g\in D$. 
\end{proposition}
\begin{proof}
The arguments are similar to those used in the second part of Proposition \ref{ddemitwists}. 
Since $g$ is a homeomorphism preserving $E$ and fixing the boundary of $P$, 
$g$ sends $\Sigma_{0,3\cdot 2^k}$ into itself, for every $k$. 
By induction on $k$, the action of $g$ on the set of boundary components of $\Sigma_{0,3\cdot 2^k}$ 
is the same as that of a product  $d_k$ of half-twists along curves of level $\leq k-1$, for any $k\geq 1$. 
The induction step follows from the fact that 
 $\Sigma_{3\cdot 2^{k+1}}-\Sigma_{3\cdot 2^{k}}$ is a union of disjoint 
 pair of pants. Then the restriction of $g$ to each such pair of pants either sends 
 the new level $k+1$ boundary components of into themselves, or else permutes them. 

Now, the element $d=d_1d_2\cdots d_k\cdots $ belongs to $\widehat{D}$ and the action of 
$d^{-1}g$ on the set of curves of $E$ is trivial.  Therefore, as above, $d^{-1}g\in \widehat{D[2]}$, so that 
$g\in \widehat{D}$, as claimed. 

The second assertion follows from Proposition \ref{ddemitwists}. 
\end{proof}

\section{The complex of decompositions of $\surf$}\label{decompositions}

\subsection{Geometric interpretation of 2-cells}
We say that two curves of a pants decomposition $F$ are \textit{adjacent} if they bound the same pair of pants in $F$. 

\begin{proposition}
 Let $F_1,F_2,F_3$ be three vertices of $\C$ such that for all $i \in {1,2}$, $F_i$ and $F_{i+1}$ are joined by an edge in $\C$. Then there exists a unique 2-cell in $\C$ containing $F_1,F_2,F_3$ as vertices. 
\end{proposition}

\begin{proof}
 There are three possible cases :
\begin{itemize}
 \item If $d(F_1,F_3)=1$, then $F_1,F_2,F_3$ are vertices of a triangular 2-cell;
 \item If $d(F_1,F_3)=2$, there are two possible cases. Let $m$ and $m'$ be the elementary moves along the curves $c$ and $c'$, which are represented by the edges $(F_1,F_2)$ and $(F_2,F_3)$, respectively;
 \begin{itemize}
  \item If $c$ and $c'$ are not adjacent in the decomposition $F_2$, the associated elementary moves have disjoint supports and they commute. Hence, $F_1,F_2,F_3$ belong to a unique squared 2-cell;
  \item If $c$ and $c'$  are adjacent, the associated elementary moves do not commute. Let $F_0$ be the decomposition obtained from $F_1$ by applying the elementary move $m'$ and $F_4$ the decomposition obtained from $F_3$ by applying $m$. Then $d(F_0,F_4)=1$. 
 \end{itemize}

\end{itemize}

\end{proof}

\subsection{Action of the mapping class group of the complex}\label{sectionactioninf}
\begin{lemma}\label{lemactioninjinf}
 The extended mapping class group $\Bext$ acts by automorphisms on the 
complex $\C$. This action is transitive on the set of vertices of $\C$. 
Moreover, the natural map $\Psi : \Bext \rightarrow \AC$ is injective.
\end{lemma}

\begin{proof}
The first part of this result is a weak version of Proposition 5.4 
of \cite{universalmcg}. We prove the injectivity part as follows. We take
an element $g \in \Bext$ and assume $g$ is non trivial. We want to show 
that we can find a vertex $F$ of $\C$ such that $g\cdot F\neq F$.

We first consider the special case where $g \in \Bd$. Let $\Sn$ be a 
quasi-admissible sub-surface of $\surf$ such that $n \geq 5$ and $g$ 
preserves the trace of  pants decompositions and seams on $\surf-\Sn$. 

 If $\Sn \neq g(\Sn)$, there is a curve $c$ of the standard 
decomposition $E$ such that $c$ is a boundary component of 
$\Sn$,  $g(c)$ is a  boundary component of $g(\Sn)$, $g(c)$ is still 
in $E$ but $g(c)$ is different from $c$. 
Let $F$ be a decomposition obtained from $E$ by 
an elementary move on $c$. Then $g \cdot F \neq F$. 

If $\Sn = g(\Sn)$, 
then the restriction of  $g$ to $\Sn$ is a non-trivial element 
of the mapping class of $\Sn$. 
The injectivity part of the main theorem of \cite{margalit} 
says that there exists a pants decomposition $F_n$ of $\Sn$ 
such that $g\cdot F_n \neq F_n$. Let $F$ be the decomposition 
of $\surf$ which coincide with $E$ outside $\Sn$ and with $F_n$ inside 
$\surf$. Then $g\cdot F \neq F$.

 Now, we turn to the general case $g \in \Bext$. We can write, following Proposition \ref{normal}: 
$$g=f \cdot \prod_{i=0}^{\infty} d_{c_i}^{p_i}  $$ 
where $f \in \Bd$ and the  sequence of curves $c_i$ belong to $E$ and eventually leave every compact sub-surface. 
If $f$ is non trivial, the previous case leads to the 
conclusion. If $f$ is trivial, then there exists 
$i \in \N$ such that $p_i \neq 0$. 
Let $F$ be a decomposition obtained from $E$ by an elementary move applied on 
the curve $c_i$. Then $g \cdot F \neq F$.

\end{proof}

\section{The automorphism group of $\C$}\label{sectionautomorphismes}
In this section, we introduce a graph whose vertices represent the  
neighbours of the vertex $E$ in $\C$. We will 
prove that the automorphism group of $\C$ acts 
naturally on this graph. We first consider the following general result.

\begin{lemma}\label{lemuniciteauto}
Let $\phi$ and $\phi'$ be two automorphisms of $\C$ such that 
for any decomposition $F$ joined by an edge to $E$ in $\C$, 
we have $\phi(F)=\phi'(F)$. Then $\phi=\phi'$.
\end{lemma}
\begin{proof}
 The proof is based on the geometric characterization of the 2-cells 
of $\C$. We denote by $d$ the combinatorial distance on the 1-skeleton 
of $\C$. We prove that $\phi(F) = \phi'(F)$ for any vertex $F$ of $\C$, 
by induction on the distance between $E$ and $F$. 
We have:
 
\begin{enumerate}
 \item First, $\phi(E)=\phi'(E)$.  Indeed, let $E,F$ and $F'$ be the vertices of a triangular 2-cell. 
 There exists precisely one more triangular 2-cell sharing two vertices with the former cell, say 
 of vertices $E,F$ and $ F''$.  By direct inspection the three vertices $F,F'$ and $F''$  uniquely determine the 
 fourth vertex of this configuration, namely any vertex $H$ such that 
 both $H,F,F'$ and $H,F,F''$ are triangular 2-cells of $\C$ coincides with $E$.  
 Therefore any automorphism of $\C$ which fixes $F,F'$ and $F''$ should also fix $E$. 
   \item For any $F$ at distance one from $E$, $\phi(F)=\phi'(F)$;
 \item By the induction hypothesis, assume that $\phi$ and $\phi'$ 
coincide on the ball of center $E$ and radius $n$ in $\C$, with $n \geq 1$. 
Let $G$ be any vertex at distance $n+1$ of $E$. We consider a path 
$p=H_0=E,H_1,...,H_{n-1},H_n,G$ of length $n+1$ joining $E$ to $G$. The 
vertices $H_{n-1},H_n,G$ define a unique 2-cell of $\C$. 
This 2-cell contains at least one fourth vertex $G'$ at 
distance less than or equal 
to $n$ from $E$. Indeed otherwise this 2-cell were triangular, 
so $G$ would be at distance less or equal to $n$ from $E$, 
contradicting our assumptions. Hence, $\phi(G)$ is necessary the 
vertex of the unique 2-cell of $\C$ defined by $\phi(G'),\phi(H_{n-1}),
\phi(H_n)$ which is joined by an edge to $\phi(H_n)$ and different from  
$\phi(H_{n-1})$. By uniqueness, $\phi(G)=\phi'(G)$.
\end{enumerate}

\end{proof}

\subsection{The link of a decomposition}\label{sectionlinkinf}
For any vertex $F$ of $\C$, we define the {\em link}  
$\LF$ to be a graph whose set of vertices consists of those vertices in 
$\C$ which are adjacent to $F$. 
Then two vertices are connected  within  $\LF$  by an {\em (A)-edge} when 
they lie together with $F$ in the same pentagonal 2-cell of $\C$, 
and by a {\em (B)-edge} when they belong to the same 
triangular 2-cell in $\C$. 
By the geometric interpretation of 2-cells in $\C$, this bi-colored 
link graph is well defined.

Further, denote  by $\LFA$ the 
graph obtained from $\LF$ by collapsing each (B)-edge (along with 
its endpoints) to a vertex. 

Let $\Sn$ be an $n$-holed sphere and $F_n$ a pants 
decomposition of $\Sn$, viewed as vertex of the 
complex $\mathcal C_{\mathcal P}(\Sn)$. 
Then one defines the {\em restricted link} $\LFn$ as above, but using only the 
neighbours in $\mathcal C_{\mathcal P}(\Sn)$.

Now, for every vertex $F$ of $\C$ and any curve $c$ in $F$, 
we denote by $V_c(F)$ the subset of those vertices of $\LF$ 
corresponding to the decompositions obtained  
by a single elementary move on $c$. Thus, the set of vertices of 
$\LF$ is the disjoint union of all 
$V_c(F)$, where $c$ runs through the set of curves of $F$. 

We obtain directly from the definitions the following lemma :

\begin{lemma}
 Let $\phi$ be any automorphism of $\C$. Then $\phi$ induces an 
isomorphism $\phi_{*,F} : \LE \rightarrow \LF$ where $F = \phi(E)$. 
Moreover, $\phi_{*,F} $ preserves the (A)- and (B)-edge types.
\end{lemma}

In particular, this lemma says that any link $\LF$ is isomorphic to $\LE$.

\subsection{Structure of $\LE$}\label{sectionstructurelink}
We give a description of the graph $\LE$ based on the geometric 
characterization of 2-cells in $\C$. 

\subsubsection*{(A)-edges}

Given two curves $c$ and $c'$ in $E$ and any two distinct vertices 
$a \in V_c(E)$ and $b \in V_{c'}(E)$, there exists an (A)-edge 
joining the vertices $a$ and $b$ in $\LE$ if and only if $c$ and $c'$ 
are {\em adjacent} in $E$, that is $c$ and $c'$ 
bound the same pair of pants in $E$.

\subsubsection*{(B)-edges}

Given two curves $c$ and $c'$ in $E$ and two distinct vertices 
$a \in V_c(E)$ and $b \in V_{c'}(E)$, if there exists a 
(B)-edge joining the vertices $a$ and $b$ in $\LE$ then $c=c'$.

\subsubsection*{Structure of $V_c(E)$}

Let $c$ be any curve in $E$. We apply an elementary move on $c$ 
such that the resulting curve $c_0$ has only one component in the 
visible side of $\surf$. Denote by $F_{c_0}$ the decomposition obtained 
from $E$ by this elementary move. 

Any other curve intersecting $c$ twice and which is disjoint 
from the other curves in $E$ can be obtained from 
$c_0$ by applying $k$ half-twists along $c$, where 
$k$ is some non-zero integer. 
We denote by $c_k$ the curve obtained from $k$ 
half-twists along $c$ on $c_0$ and $F_{c_k}$ the resulting decomposition 
of $\surf$. Then we have a canonical identification:  
\[ V_c(E)=\{F_{c_k}, k \in \Z\}.\] 
Furthermore,  $F_{c_k}$ and $F_{c_{m}}$ are joined by an edge 
(which is necessarily of $B$-type) if and only if $|k-m|=1$.

We remark that a vertex in $V_c(E)$ represents a homotopy class of a 
non-trivial curve in $\Sigma_{0,4}$ and that two vertices in 
$V_c(E)$ are joined by an edge if the associated 
curves on $\Sigma_{0,4}$ have geometric intersection equal to 2. 
This graph has been considered before, in relation with the curve complex 
of $\Sigma_{0,4}$. It was already known to Max  Dehn that this graph is 
isomorphic to the Farey graph $\tau_*$ (see \cite{De} and \cite{luo}, section 3.2).  

Note that $V_c(F)$ has a similar structure, for any $F$. By choosing   
$g\in \Bd$ such that $F=g(E)$, we derive an identification 
of $V_c(F)$ with $V_{g^{-1}(c)}(E)$. This identification is not unique, as 
we can alter $g$ by an element of $D$.

\subsection{Structure of $\LEA$}
The image of $V_c(E)$ by the collapsing map 
$\LE\to \LEA$  is a single point. Thus it will make sense to speak about 
the vertex $V_c(E)$ of $\LEA$. This provides a one-to-one 
correspondence between the set of vertices of $\LEA$ and the 
set of curves in $E$. Moreover, two vertices of $\LEA$ are connected by 
an edge if and only if they represent two adjacent curves in $E$.
Thus $\LEA$ consists of triangles indexed by the pair of pants occurring in $E$. 
We will often consider the dual of $\LEA$, whose vertices correspond to pair of pants in $E$ 
and edges to curves in $E$, which is therefore combinatorially isomorphic to the binary tree $\mathcal T$.

\subsection{Action of half-twists on the links}\label{sectionactionD}

Let $d_c$ be a half-twist whose support is a pair of pants bounded by 
the curves $a,b,c$ of $E$, such that $d$ permutes $a$ and $b$. Then 
$d_c$ acts on $\LEA$ by reversing the cyclic order of the vertices of the 
triangle associated to the support of $d_c$ in $\LEA$ and fixing the vertex 
$V_c(E)$. Moreover $d_c$ acts on $\LE$ and preserves the type of edges in $\LE$. 
Hence, $d_c$ acts on the union of the sets $V_{c'}(E)$ over the curves $c'$ 
of $E$. We distinguish three cases in the description of  
the action on  $V_{c'}(E)$, 
according to whether $c'$ is adjacent to one 
of the curves $a,b,c$ or not :

\begin{enumerate}

 \item Let $c'$ be a curve from $E$ contained in $supp(d_c)$ and different from $c$. 
 The half-twist $d_c$ sends $c'$ into a different curve $c''$ of $E$. 
 The induced map $V_{c'}(E)\to V_{c''}(E)$ is given by 
$F_{c'_k} \mapsto F_{c''_{k}}$,  for $k \in \Z$.

 \item Further $V_c(E)$ is stable by $d_c$ and $d_c$ acts on this set by the 
map $F_{c_k} \mapsto F_{c_{k+1}}$, $k \in \Z$.

 \item If $c'$ is a curve belonging to the component of $\surf \setminus\{c\}$ 
which does not contain $a$ and $b$, then $d_c$ fixes 
$V_{c'}(E)$ point-wise.

\end{enumerate}

\begin{figure}[ht!]
\begin{center}
\includegraphics[width=0.7\textwidth]{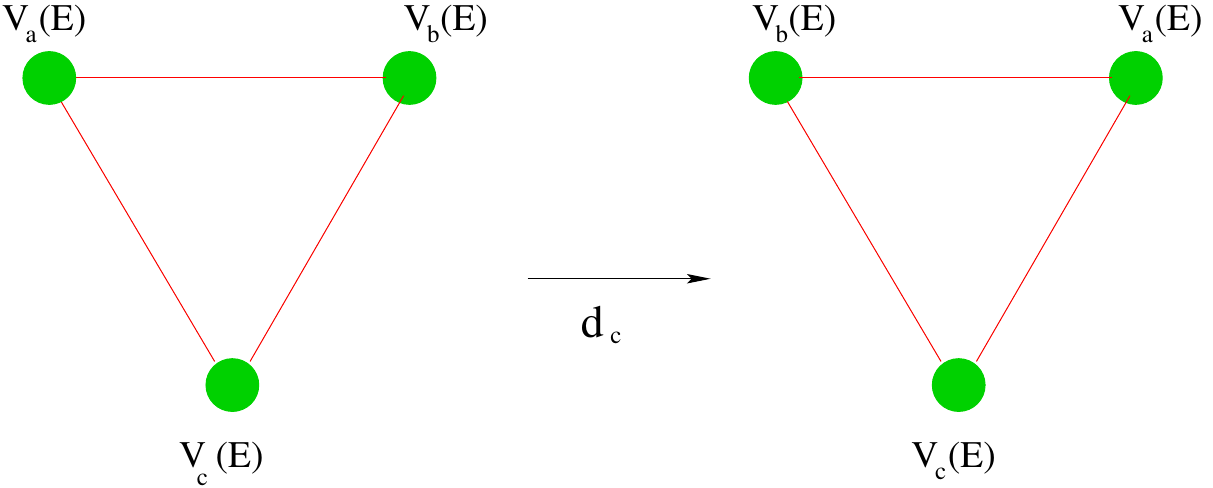}
\caption{\label{demitwistlink}Action of the half-twist $d$ on the link $\LEA$}
\end{center}
\end{figure}

\subsection{Action of $\T$ on the links}\label{sectionactionT}

  Let $t$ be an element of the Thompson group $\T$. Then $t$ induces 
an isomorphism $t_{*,F} : \LE \rightarrow \LF$, where $F = t(E)$. If $c$ 
is a curve of $E$, there exists a curve $c'$ of $F$ such that 
$t_{*,F}(V_c(E))=V_{c'}(F)$.  The mapping class $t$ preserves the visible 
face of $\surf$ from Proposition \ref{propvisibleT}. For any 
$k \in \Z$, we have  then $t_{*,F}(F_{c_k})=F_{c'_k}$.

\subsection{Characterization of automorphisms of $\LE$}\label{sectioncarac}

We deduce some conditions for an 
isomorphism $\LE \rightarrow \LE$ to be induced by an element of $\Bext$.

\begin{lemma}\label{lemcarac}
Let $f \in \Bext$ such that  
$\phi(E)=E$, where  $\phi=\Psi(f) \in \AC$. Let $c$ be a curve in $E$ and $c'$ 
be the curve of $E$ such that $\phi(V_c(E))=V_{c'}(E)$. Then there 
exists $p \in \Z$, such that for all $k \in \Z$, 
we have $\phi(F_{c_k})=F_{c'_{\varepsilon(f)k + p}}$, where 
$\varepsilon(f)=1$, if $f$ preserves the orientation and $\varepsilon(f)=-1$, otherwise. 
 \end{lemma}
\begin{proof}
Assume we choose a set of identifications $\alpha_{c,F}:V_c(F)\to \Z$, for all $F$ and 
$c\in F$, which coincides with the  canonical identification above for $V_c(E)$.
Set ${\rm Perm}(\Z)$ for the group of bijections of $\Z$.  
We have a map $\mu_{c,F} : \Bext \rightarrow \rm{Perm}(\Z)$ 
defined by $\mu_{c,F}(g)(k)= \alpha_{g(c),g(F)}\circ g_*\circ \alpha_{c,F}^{-1} (k)$, 
where $k\in \Z$, $c\in F$ and $g_*:V_c(F)\to V_{g(c)}(g(F))$ is the map induced by $g\in \Bext$. 
This map is a  twisted group homomorphism, in the sense that 
$\mu_{c, g\circ f(F)}(g\circ f)=\mu_{f(c),f(F)}(g)\circ \mu_{c,F}(f)$, if $f, g\in \Bext$. 
We set $\mu_c(f)=\mu_{c,E}(f)$, when $c\in E$ and $f(E)=E$. 
Note that $\mu_c(f)$ is independent on the choice of $\alpha_{c,F}$ extending the canonical identification, if $f(E)=E$, and  we have $\phi(F_{c_k})=F_{c'_{\mu_c(f)(k)}}$, where $\phi=\Psi(f)$.  

From section \ref{sectionactionT}, we have $\mu_{c,E}(t)=id_\Z$ for any 
$t \in \T$. From section \ref{sectionactionD} 
if $d$ is a half-twist then 
$\mu_c(d)$ is either the identity $id_{\Z}$, 
or the translation $\mu_c(d)(k )= k+1$. Also $\mu_c(i_R)(k)=-k$. 
Since any element of $\Bext$ is a product of an element of $T$ along with 
infinitely many half-twists accumulating at infinity, $\mu_c(f)$ has the required form when 
$f(E)=E$.  
 \end{proof}

\begin{lemma}\label{even}   
Let $f \in \Bext$ such that  
$\phi(E)=E$, where  $\phi=\Psi(f) \in \AC$.  
Let $P'$  and $P''$ be two almost admissible pairs of pants which are 
bounded by the curves $a,a',a''$ and $a,b,b'$, respectively. 
We assume that:
\begin{enumerate}
\item $a$ is closer to the fundamental pair of pants than $a'$ and $a''$;
\item the vertices $V_a(E),V_{a'}(E), V_{a''}(E)$ and $V_{b}(E)$  of $\LEA$ are point-wise fixed by $\phi$.
\end{enumerate}
Then $\mu_a(f)(k)=\varepsilon(f)k+p$,  where $p$ is even.
\end{lemma}
\begin{proof}
We reduce ourselves to the case when $f$ preserves the orientation, by composing if needed with the symmetry $i_R$.  
Choose an orientation preserving  rigid mapping class $h$ such that   $h(P'')=P$ and 
$h(a)$ is a boundary curve of the fundamental pair of pants $P$, whereas $h(a')$ and $h(a'')$ are not.
Then $\widetilde{f}=h\circ f\circ h^{-1}$ still fulfills $\widetilde{f}(E)=E$. 
Since two vertices of the triangle $V_{a}(E), V_{b}(E), V_{b'}(E)$ of $\LEA$ are fixed by $\phi$, the third one will 
also be fixed, i.e. $\phi(V_{b'}(E))=V_{b'}(E)$.   This implies that $f(b)=b$ and $f(b')=b'$, so that $\widetilde{f}$ preserves each boundary component of $P$.  
One can compose $\widetilde{f}$ with a product $g$ of Dehn twists along closed curves in $P$ and 
possibly half-twists along $h(b)$ and $h(b')$  
such that $\widetilde{f}\circ g$ is  identity on $P$. Hence $\widetilde{f}\circ g\in \widehat{D}$, by 
Proposition \ref{Dstabilizer}.  Now, the triangle  with vertices 
$V_{h(a)}(E)$, $V_{h(a')}(E)$ and  $V_{h(a'')}(E)$ is pointwise fixed by $\Psi(\widetilde{f}\circ g)$. It 
follows that in the expression of $\widetilde{f}\circ g$, and hence of $\widetilde{f}$,  as a  product of 
half-twists along curves in $E$ going to infinity the half-twist along $h(a)$ occurs with an even exponent.  
The claim is a consequence of  the fact that $\mu_{h(a)}(h\circ f\circ h^{-1})=
\mu_a(h)\circ\mu_a(f)\circ (\mu_a(h))^{-1}$. 
\end{proof}

\section{Compact sub-surfaces}\label{cascompact}
Observe first, that for all $n \geq 5$ and any  admissible sub-surface $\Sn$ 
of $\surf$ the inclusion  $\Sn\subset \surf$ induces an embedding 
of $\mathcal{C_P}(\Sn)$ into $\C$.
In this section, we will show that, for any automorphism $\phi$ of $\C$ 
there exists $n \geq 5$, admissible sub-surfaces $\Sn$ 
and $\Sn'$ of $\surf$ and an induced isomorphism between the pants complexes 
\[\phi_n : \mathcal{C_P}(\Sn) \rightarrow \mathcal{C_P}(\Sn'),\] 
in the sense that the restriction of $\phi$ to the 
sub-complex $\mathcal{C_P}(\Sn)\subset \C$ coincides with $\phi_n$.

\subsection{Definition of the induced isomorphism}\label{induced}

Let $\phi$ be an automorphism of $\C$ and set $F=\phi(E)$. 
There exists an admissible sub-surface $\Sn '$ of level $n\geq 5$ such 
that $F$ coincides with $E$ outside $\Sn '$. 
The decomposition $F$ induces a pants decomposition $F_n'$ of $\Sn'$.

\begin{description}
 \item[Definition of $\Sn$]
\end{description}

Consider the set $W'$ of vertices of $\LFA$ of the form $V_{c'}(F)$, for all 
curves $c'\in F_n'$. The set $W'$ spans a connected tree in  the dual tree of $\LFA$,  
which is isomorphic to the dual tree of $F_n'$. The automorphism 
$\phi^{-1}$ induces an isomorphism $\phi^{-1}_{*,F}$ from  $\LFA$ 
to $\LEA$.  
Therefore $W=\phi^{-1}_{*,F}(W')$ is a set of vertices of $\LEA$ of 
the form $V_c(E)$, where $c$ belongs to some finite subset $E_n$ 
of curves of $E$. Since $W$ also spans a connected subtree of the dual of $\LFA$, 
the curves of $E_n$ form a pants decomposition of a
connected  admissible sub-surface $\Sn$ of $\surf$.  

\begin{figure}[ht!]
\begin{center}
\includegraphics[width=0.7\textwidth]{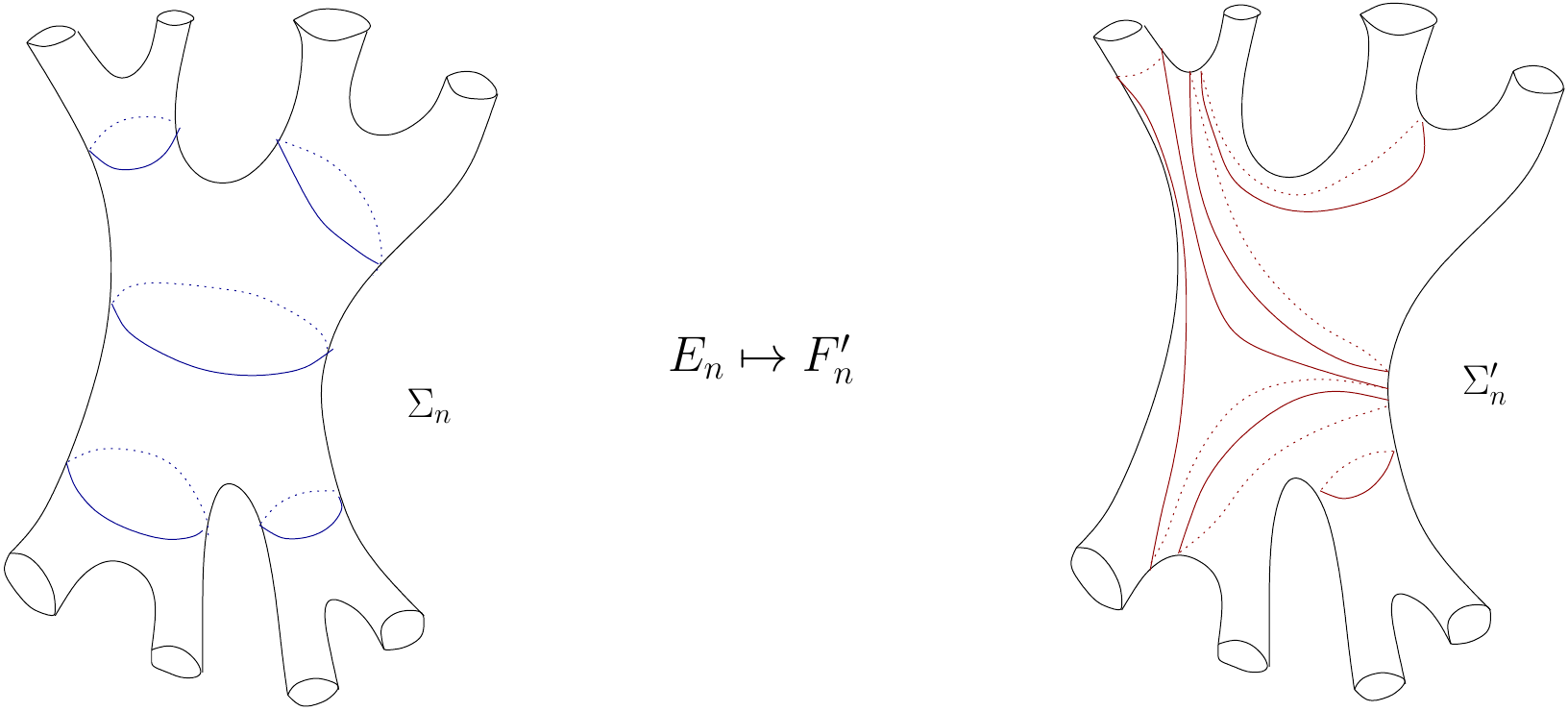}
\caption{\label{sinf3}The isomorphism $\mathcal{C_P}(\Sn) \rightarrow \mathcal{C_P}(\Sn')$}
\end{center}

\end{figure}

\begin{description}
 \item[The isomorphism $\phi_n:\mathcal{C_P}(\Sn) \rightarrow \mathcal{C_P}(\Sn')$]
\end{description}




Let $G$ be a decomposition adjacent to the  
canonical decomposition $E$ in $\C$. Assume that this decomposition is 
issued from an elementary move on a curve belonging to the interior 
of the sub-surface $\Sn$. 
Then $\Sn$ contains a support of $G$. Moreover, by construction of $\Sn$, 
$\phi(G)$ is issued from an elementary move in $F$ on a curve of $\Sn'$ 
and induces a pants decomposition of $\Sn'$. 

By induction on the number of moves, the image by $\phi$ of any decomposition $G$ of $\surf$ 
obtained by applying finitely many elementary moves on 
curves  lying in the interior of $\Sn$ is a 
decomposition of $\surf$ whose support is included in $\Sn'$. In other 
words, $\phi(G)\cap \Sn'$ is a well-defined vertex of $\mathcal{C_P}(\Sn')$. 
It remains to note that the pants decomposition of $\Sn'$ 
is obtained by finitely many elementary moves on the 
decomposition of $\Sn$. This correspondence  provides a cellular isomorphism 
$\phi_n:\mathcal{C_P}(\Sn) \rightarrow \mathcal{C_P}(\Sn')$. 

\begin{lemma}\label{lemcompactn}
 Let $\phi$ be an element of $\AC$ and $\Sn$, with $n\geq 5$, be a sub-surface as above. There exists an 
element $g \in \Bd$ such that $\phi$ and $\Psi(g) \in \AC$ coincide 
on the restriction of the link $\LE$ to vertices of $\mathcal C_{\mathcal P}(\Sn)$. 
\end{lemma}

\begin{proof}
Let  $E_n'$ denote the trace of the canonical decomposition on $\Sn'$. 
The dual trees of $E_n$ and of $E_n'$ are not necessarily isomorphic, 
in general. Let $\partial E_n$ and $\partial F_n'$ 
denote the set of those curves in $E_n$ and $F_n'$ lying in the boundary of  
$\Sn$ and $\Sn'$, respectively. The map $\phi_*$ at the level 
of $\LEA$ induced an isomorphism between the dual trees of $E_n$ and 
$F_n'$, sending therefore  leaves onto leaves and thus 
 $\partial E_n$ onto $\partial F_n'$. We can choose then a homotopy class of a 
homeomorphism $f$ sending $\Sn$ on $\Sn'$ such that 
the restriction of $f$ to $\partial E_n$ acts combinatorially  as 
$\phi_*$. This homeomorphism induces an isomorphism between the pants complexes 
$\overline{f} : \mathcal{C_P}(\Sn) \rightarrow \mathcal{C_P}(\Sn')$. 
By left-composing $\phi_n$ by ${\overline{f}}^{-1}$, 
we define an automorphism $\phi_{n,f}$ of   $\mathcal{C_P}(\Sn)$.

Now, Margalit's rigidity theorem (\cite{margalit})  
states that there exists an 
element $F_{n,f}$ of the extended mapping class group of $\Sn$  
whose action on $\mathcal{C_P}(\Sn)$ is the automorphism 
$\phi_{n,f}$. 

We now consider the map $f \circ F_{n,f} : \Sn \rightarrow \Sn'$ and 
extend it rigidly to the surface $\surf$. Let $g_{n,f}$ be the 
asymptotically rigid mapping class  of $\surf$ resulting from this extension. 
By construction, $\Psi(g_{n,f})$ coincides with $\phi$ on the 
restricted link $\LEn$. 
\end{proof}

\subsection{Extension of the isomorphisms of the link}

We consider an extension of  
the map $\mu_c: \Bext\to {\rm Perm}(\Z)$, defined in the proof of Lemma \ref{lemcarac}, 
from $\Bext$ to $\AC$. 
As the most general case will not be needed later, we restrict ourselves to the subgroup 
$\AC_E$ of those 
$\phi\in \AC$ satisfying $\phi(E)=E$.  For every $c\in E$ let $c'\in E$ be the curve  
with the property  that $\phi(V_c(E))=V_{c'}(E)$.  Then $\mu_c$ is 
defined by $\phi(F_{c_k})=F_{c'_{\mu_c(\phi)(k)}}$. It is clear that 
$\mu_c: \AC_E\to {\rm Perm}(\Z)$ is a twisted homomorphism, namely 
$\mu_c(\phi_1\circ\phi_2)=\mu_{\phi_2(c)}(\phi_1)\circ \mu_c(\phi_2)$. 

Lemma \ref{lemcompactn} allows us to generalize some local characterizations 
described in section \ref{sectioncarac} to all automorphisms of $\C$. 
More 
precisely, we have the following result :

\begin{lemma}\label{lemquasirigide2bis}
Consider  $\phi \in \AC$ such that $\phi(E)=E$. 
Let $P'$  and $P''$ be two almost admissible pairs of pants which are 
bounded by the curves $a,a',a''$ and $a,b,b'$, respectively. 
We assume that:
\begin{enumerate}
\item $a$ is closer to the fundamental pair of pants than $a'$ and $a''$;
\item the vertices $V_a(E),V_{a'}(E), V_{a''}(E)$ and $V_{b}(E)$  of $\LEA$ are point-wise fixed by $\phi$.
\end{enumerate}
 Then $\mu_a(\phi)(k)=\varepsilon(\phi)k+p$,  where $p$ is even and 
$\varepsilon(\phi)\in \{-1,1\}$.
\end{lemma}
\begin{proof}
We choose a large enough admissible sub-surface $\Sn'$ as in section \ref{induced} so that the restricted link $\LEn$ 
contains $V_a(E)$, $V_{a'}(E)$, $V_{a''}(E)$, $V_{b}(E)$ and $V_{b'}(E)$. 
From Lemma \ref{lemcompactn}, there is an element $f \in \Bext$ such that the automorphisms $\Psi(f)$ and $\phi$ coincide on the restricted link $\LEn$.  
Thus, the  
properties valid for $\Psi(f)$ which were given in 
Lemma \ref{even} are also valid for $\phi$.
\end{proof}

\subsection{Proof of the theorem}

Let $\phi$ be an automorphism of $\C$. We shall construct an 
element of $\Bext$ which acts as $\phi$ on $\C$.

According to Lemma \ref{lemcompactn}, there is an 
element $g$ of $\Bd$ which is quasi-rigid outside a quasi-admissible sub-surface 
$\Sigma_{0,n}$, $n\geq 5$, of $\surf$, such that $\phi(E)$ coincide with $E$ outside 
$g(\Sigma_{0,n})$. Moreover, the automorphisms $\phi$ and $\Psi(g)$ 
coincide on the restricted link $\LEn=\LE\cap \mathcal{C_P}(\Sn)$ on $\Sigma_{0,n}$. 
We can assume that $\Sigma_{0,n}$ contains the fundamental pair of pants. 

We now consider the automorphism $\phi \circ \Psi(g^{-1})$. Let $a$ be 
a boundary component of $\Sn$ (so that $a$ is a curve of $E$) and $P_a$ be the 
 pair of pants of $\surf$ with boundary in $E$ such that $P _a\cap \Sn = a$. Then $P_a$ has three 
boundary components $a,a',a''$ defining a triangle 
$V_a(E),V_{a'}(E),V_{a''}(E)$ in $\LEA$. Both the triangle and its vertex $V_a(E)$ are invariant by 
$\phi \circ \Psi(g^{-1})$. By possibly composing $g$ with $d_a$ we can assume that 
this triangle in $\LEA$ is point-wise invariant, while $\phi \circ \Psi(g^{-1})$ 
still acts as  identity on $\LEn$, because $\Psi(d_a)$ is  identity on $\LEn$.
 Observe that $\varepsilon(\phi\circ \Psi(g^{-1}))=1$, as this automorphism is  identity on $V_b(E_n)\subset  
\mathcal{C_P}(\Sigma_{0,n})$, when $b$ is a curve of  $E\cap \Sigma_{0,n}$ not homotopic to a boundary component
and adjacent to $a$. 

According to 
Lemma \ref{lemquasirigide2bis}, there exists  an even $p\in \Z$ such that 
for any $k \in \Z$, $\mu_a(\phi\circ \Psi(g^{-1}))(k)=k+p$, 
because the vertices $V_a(E),V_{a'}(E),V_{a''}(E)$ and $V_b(E)$ are 
point-wise invariant by $\phi \circ \Psi(g^{-1})$. 
 Let then $g'=d_a^p \circ g$. Due to the 
description of the action of an half-twist on the link, we know that $\phi$ 
and $\Psi(g')$ coincide on $\LE\cap \mathcal{C_P}(\Sigma_{0,n}\cup P_a)$.  

By induction, there exists  an infinite product $d$ of half-twists  along curves in $E$ 
accumulating to infinity  such that the automorphisms
$\Psi(d\circ g)$ and $\phi$ coincide on all vertices adjacent 
to $E$ in $\C$. Now $d\circ g\in \Bext$ and these two automorphisms are the same by  
Lemma \ref{lemuniciteauto}. Thus $\Psi$ is onto. 

\vspace{0.2cm}

{\bf Acknowledgements}. The authors are indebted to 
Javier Aramayona, Ariadna Fossas, Pierre Lochak, Athanase Papadopoulous 
and Vlad Sergiescu for useful discussions  and to the referee for pointing out several inaccuracies 
and for his/her suggestions which helped improving the presentation. 
They were supported by the ANR 2011 BS 01 020 01 ModGroup. 
Part of this work was done while the first author visited Erwin Schr\"odinger 
Institute which he wants to thank warmly for hospitality and support.

\end{document}